\documentclass[dvipdfmx]{amsart}
\usepackage{amssymb, amsmath}
\usepackage{hyperref}

\newtheorem{thm}{Theorem}[section]

\newtheorem{lemma}[thm]{Lemma}

\newtheorem{prop}[thm]{Proposition}

\theoremstyle{definition}
\newtheorem{dfn}[thm]{Definition}
\newtheorem{exa}[thm]{Example}
\theoremstyle{remark}
\newtheorem*{remark}{Remark}

\def\QQ{\mathord{\mathbb{Q}}}

\def\ZZ{\mathord{\mathbb{Z}}}

\let\oldbigwedge\bigwedge
\renewcommand{\bigwedge}{\mathord{\oldbigwedge}}

\title{An elementary construction of the ring of dual $K$-$Q$-cancellation property}
\author{Shinsuke Iwao}
\address{Faculty of Business and Commerce, Keio University, Hiyosi 4–1–1, Kohoku-ku, Yokohama-si, Kanagawa 223-8521, Japan.}
\email{iwao-s@keio.jp}
\thanks{
This paper will be published in \textit{RIMS K\^{o}ky\^{u}roku Bessatsu}, which accepts papers written in both Japanese and English.
This paper was originally written in Japanese and was translated into English for publication. The author is grateful to the anonymous referee who encouraged him to prepare an English version. This work was partially supported by the Grant-in-Aid for Scientific Research 19K03605, 22K03239, 23K03056}  
\begin{document}

\maketitle

\begin{abstract}      
This paper presents an elementary introduction on $K$-theoretic $Q$-functions, which were introduced by Ikeda and Naruse in 2013. These functions, which serve as $K$-theoretic analogs of Schur $Q$-functions, are known to possess combinatorial and algebraic constructions. 
In a 2022 paper, the author introduced ``$\beta$-deformed power-sums" to provide a simpler, more algebraic construction of these functions. 
Since the original approach relies on fermionic operators and vacuum expectation values, this paper presents a more accessible, purely algebraic treatment, following the exposition of Schur $Q$-functions in Macdonald's standard textbook.
We also show that the algebra of dual $Q$-cancellation property with integer coefficients is generated by dual $K$-$Q$-functions associated with an odd row partition.
\end{abstract}

\section{Introduction}

\subsection{About this paper}

This paper aims to give an elementary introduction on results in the paper \cite{iwao2022,iwao2023} of the author.
The \emph{$K$-theoretic $Q$-functions} (or \emph{$K$-$Q$ functions}) were introduced in 2013 by Ikeda and Naruse~\cite{IKEDA201322}.  
These functions were introduced in order to represent Schubert classes in the $K$-theory of the Lagrangian Grassmannian,
and can be regarded as $K$-theoretic analogues of the \emph{Schur $Q$-functions} (see~\cite[§III-8]{macbook}).

The $K$-$Q$ functions are known to have many interesting properties, including a \emph{Pfaffian formula}~\cite{HUDSON2017115} and a \emph{combinatorial expression in terms of Young diagrams}~\cite{IKEDA201322}.  
In parallel with the classical Schur $Q$-functions, they also admit an algebraic construction using the power-sum and the $q$-function~\cite[III,\S 8]{macbook}.

In a previous work~\cite{iwao2022}, the author introduced a certain generalization of the power-sum called the \emph{$\beta$-deformed power-sum} (or simply, the \emph{$\beta$-power-sum}), and constructed the $K$-$Q$ functions in terms of these $\beta$-power-sums.
While the original paper~\cite{iwao2022} relies on vacuum expectation value expressions of fermionic operators, which may be difficult to read without sufficient background, this paper aims to present a purely algebraic construction in the style of Macdonald's textbook~\cite{macbook}.

\subsection{Background}

Throughout this paper, we let \( \beta \) be a formal parameter. 
The \textit{\( K \)-theoretic sum} \( \oplus \) is the binary operation defined by
\[
x \oplus y = x + y + \beta x y.
\]
When \( \beta = 0 \), this reduces to the usual addition \( + \). The additive inverse of \( x \) with respect to \( \oplus \) is given by
\[
\overline{x} = \frac{-x}{1 + \beta x}.
\]

Let \( x = (x_1, x_2, \dots) \) be an infinite sequence of variables, and let \( p_i = p_i(x) = x_1^i + x_2^i + \cdots \) denote the power sum.
The Schur \( Q \)-functions were introduced by Schur in 1911 to describe projective representations of the symmetric and alternating groups. These functions form a \( \mathbb{Q} \)-basis of the ring \( \mathbb{Q}[p_1, p_3, p_5, \dots] \), and they exhibit various algebraically interesting properties, such as a Pfaffian formula and a combinatorial expression via shifted Young diagrams.
The algebra \( \mathbb{Q}[p_1, p_3, p_5, \dots] \) is characterized as a $\QQ$-subalgebra of symmetric functions by the following \textit{\( Q \)-cancellation property}:
\begin{center}
\( f(t, -t, x_3, x_4, \dots) \) is independent of \( t \).
\end{center}

The \textit{\( K \)-\( Q \)-cancellation property} is a \( K \)-theoretic analogue of the \( Q \)-cancellation property. 
It serves as a defining condition for the subalgebra $G\Gamma$ (to be defined below) of the completed ring \( \widehat{\Lambda} \) of symmetric functions.

Let \( \Lambda \) be the \( \mathbb{Q}(\beta) \)-algebra of symmetric functions in \( x_1, x_2, \dots \), and let \( \mathcal{P} \) be the set of all partitions.
The algebra $\Lambda$ is generated by \textit{Schur functions} $s_\lambda(x)$ parametrized by a partition $\lambda\in \mathcal{P}$.
As a $\QQ(\beta)$-vector space, $\Lambda$ is decomposed as $\Lambda=\bigoplus_{\lambda\in \mathcal{P}}\QQ(\beta)\cdot s_\lambda(x)$.

For $\lambda\in \mathcal{P}$, let $\ell(\lambda)$ denote the length of $\lambda$.
Define the $\QQ(\beta)$-subspace
\[
U_n := \{ f(x) \in \Lambda \, ; \, f(x_1, x_2, \dots, x_n, 0, 0, \dots) = 0 \}
=\bigoplus_{\ell(\lambda)> n}\QQ(\beta)\cdot s_\lambda(x) ,
\]
and endow \( \Lambda \) with the linear topology where the family  \( \{ U_n \}_{n=1,2,\dots} \) forms a fundamental system of neighborhoods of $0$. The completion of \( \Lambda \) with respect to this topology is defined as a inverse limit
\( \widehat{\Lambda}=\lim\limits_{\longleftarrow n}(\Lambda/U_n) \) of the system
\[
\Lambda/U_1\leftarrow \Lambda/U_2\leftarrow \Lambda/U_3\leftarrow\cdots.
\]
%
Any element of \( \widehat{\Lambda} \) is uniquely expressed as an infinite sum
\begin{equation}\label{eq:elem_of_widehat_en}
\sum_{\lambda \in \mathcal{P}} c_\lambda s_\lambda(x), \qquad
\left(
\begin{aligned}
& c_\lambda \in \mathbb{Q}(\beta), \\
& \# \{\lambda\,|\,c_\lambda\neq 0,\,\ell(\lambda)=n\}<+\infty \text{ for each } n\geq 0.
\end{aligned}
\right)
\end{equation}

The following was introduced in \cite{IKEDA201322}:

\begin{dfn}[{\cite[Definition 1.1]{IKEDA201322}}]
A symmetric function \( f(x) \in \widehat{\Lambda} \) is said to satisfy the \textit{\( K \)-\( Q \)-cancellation property} if
\[
f(t, \overline{t}, x_3, x_4, \dots) \text{ is independent of } t.
\]
When \( \beta = 0 \), this condition reduces to the ordinary \( Q \)-cancellation property.
\end{dfn}

On the other hand, the dual version of this concept was introduced in \cite{iwao2022}:

\begin{dfn}[{\cite[\S 8]{iwao2022}}]
A function is said to satisfy the \textit{dual \( K \)-\( Q \)-cancellation property} if
\[
f(t, -t - \beta, x_3, x_4, \dots) \text{ is independent of } t.
\]
This condition also reduces to the \( Q \)-cancellation property when \( \beta = 0 \).
\end{dfn}

We define the rings \( G\Gamma \) and \( g\Gamma \) as follows:
\[
\begin{gathered}
G\Gamma := \left\{
f(x) \in \widehat{\Lambda} \;\middle|\;
\text{\( f \) satisfies the \( K \)-\( Q \)-cancellation property}
\right\}, \\
g\Gamma := \left\{
f(x) \in \Lambda \;\middle|\;
\text{\( f \) satisfies the dual \( K \)-\( Q \)-cancellation property}
\right\}.
\end{gathered}
\]

In~\cite{IKEDA201322}, Ikeda–Naruse defined the \( K \)-\( Q \)-functions \( GQ_\lambda \in G\Gamma \), and showed that they represent Schubert classes in the \( K \)-theory of the Lagrangian Grassmannian\footnote{The original paper~\cite{IKEDA201322} deals with a more general result in equivariant \( K \)-theory.}. 
The spaces $G\Gamma$ and $g\Gamma$ are related via a certain bilinear form, defined through the \textit{$K$-theoretic Cauchy kernel} (see \eqref{eq:K-cauchy} below) introduced by Nakagawa–Naruse~\cite[Definition 5.3]{nakagawa2016generating} (see also \cite[Definition 6.2]{nakagawa2017generating}).

Both \( G\Gamma \) and \( g\Gamma \) can be viewed as \( K \)-theoretic analogues of the ring \( \mathbb{Q}[p_1, p_3, p_5, \dots] \). These algebras admit explicit \( \mathbb{Q}(\beta) \)-bases constructed using the \(\beta\)-deformed power-sums \( p^{(\beta)}_n \), \( p^{[\beta]}_n \), and the \(\beta\)-deformed \( q \)-functions \( q^{(\beta)}_n \), \( q^{[\beta]}_n \), introduced in~\cite{iwao2022}.

This paper aims to present key structural features of \( G\Gamma \) and \( g\Gamma \), such as the determination of bases and finiteness properties, through purely elementary algebraic arguments.
The structure of this paper is as follows.
In Chapter~\ref{sec:KQ}, we describe the \( \mathbb{Q}(\beta) \)-algebra \( G\Gamma \), including its generators and basis. Every element of \( G\Gamma \) can be uniquely written as an infinite linear combination of either \(\beta\)-deformed power sums or \(\beta\)-deformed \( q \)-functions (Theorem~\ref{prop:expansion_of_GGamma}).
In Chapter~\ref{sec:dualKQ}, we construct an explicit \( \mathbb{Q}(\beta) \)-basis of \( g\Gamma \) (Theorem~\ref{prop:expansion_of_gGamma}) and establish a natural bilinear pairing between \( G\Gamma \) and \( g\Gamma \) induced by the Nakagawa–Naruse kernel (\S\ref{sec:dual_speces}).
In Chapter~\ref{eq:z-str_of_gGamma}, we investigate the integral structure of \( g\Gamma \) as a \( \mathbb{Z}[\beta] \)-algebra.

\section{Algebra of the \( K \)-\( Q \)-cancellation property}\label{sec:KQ}

\subsection{Preliminaries}

In this section, we introduce a \( \mathbb{Q}(\beta) \)-basis of \( G\Gamma \) (as a $\QQ(\beta)$-vector space), as well as a set of generators as a \( \mathbb{Q}(\beta) \)-algebra. 
Since elements of \( G\Gamma \) are generally expressed as infinite sums of the form~\eqref{eq:elem_of_widehat_en}, we consider a ``basis'' in the sense of infinite expansions rather than a Hamel basis in the usual sense (see Proposition~\ref{prop:expansion_of_GGamma} for the precise statement).

To deal with infinite sums, it is natural to consider an enlarged vector space that contains \( G\Gamma \) as a dense subspace. 
We define such an extension by including formal power series in the power sums \( p_i=p_i(x) \).

Consider the \( \mathbb{Q}(\beta) \)-vector space \( \widehat{\Lambda}_\sharp \) consisting of formal sums of the form
\[
\sum_{n=0}^\infty \sum_{
\substack{
|\lambda|=n
}
} c_\lambda\, p_{\lambda_1} p_{\lambda_2} \dots p_{\lambda_{\ell(\lambda)}}, \qquad (c_\lambda \in \mathbb{Q}(\beta)).
\]
The space \( \widehat{\Lambda}_\sharp \) admits a graded algebra structure where the degree of the function \( p_{\lambda_1} p_{\lambda_2} \dots p_{\lambda_{\ell(\lambda)}} \) is defined as \( |\lambda| \). 
The ring \( \widehat{\Lambda} \) is a subalgebra of \( \widehat{\Lambda}_\sharp \).
An element \( f(x) \in \widehat{\Lambda}_\sharp \) belongs to \( \widehat{\Lambda} \) if, for every \( N \geq 1 \), the specialization
\[
f(x_1, \dots, x_N, 0, 0, \dots)
\]
is a polynomial.

We now define a subalgebra \( G\Gamma_\sharp \subset \widehat{\Lambda}_\sharp \) by
\[
G\Gamma_\sharp := \left\{
f(x) \in \widehat{\Lambda}_\sharp \;\middle|\;
f \text{ satisfies the \( K \)-\( Q \)-cancellation property}
\right\}.
\]
From the definitions above, we have
\begin{equation}\label{eq:car_of_GGamma}
G\Gamma = G\Gamma_\sharp \cap \widehat{\Lambda}.
\end{equation}

\subsection{\( \mathbb{Q}(\beta) \)-basis of \( G\Gamma_\sharp \)}

The algebra \( G\Gamma_\sharp \) is a \( K \)-theoretic analogue of the algebra of Schur \( Q \)-functions
\[
\Gamma|_{\beta=0} := \left\{ f(x) \in \Lambda|_{\beta=0} \;\middle|\; 
f \text{ satisfies the \( Q \)-cancellation property} \right\}.
\]
It is known that $\Gamma|_{\beta=0}$ admits the expression
\begin{equation}\label{eq:Gamma_0}
\Gamma|_{\beta=0} = \mathbb{Q}[p_1, p_3, p_5, \dots],
\end{equation}
as shown in~\cite[\S III-8]{macbook}. 
Here, \( p_1, p_3, p_5, \dots \) are algebraically independent over \( \mathbb{Q} \). Since \eqref{eq:Gamma_0} is preserved under base change, we also have
\begin{equation}\label{eq:Gamma_beta}
\Gamma = \Gamma|_{\beta=0} \otimes_{\mathbb{Q}} \mathbb{Q}(\beta) = \mathbb{Q}(\beta)[p_1, p_3, p_5, \dots].
\end{equation}

\subsubsection{\( \beta \)-deformed power sums \( p_n^{(\beta)} \)}

We define the \textit{\( \beta \)-deformed power sum} \( p_n^{(\beta)} = p_n^{(\beta)}(x) \in \widehat{\Lambda}_\sharp \) by
\begin{equation}\label{eq:p^beta_vs_p}
p^{(\beta)}_n(x) = p_n\left( \frac{x}{1 + \frac{\beta}{2}x} \right)
= \sum_{k=0}^\infty \binom{-n}{k} \left( \frac{\beta}{2} \right)^k p_{n+k}(x),
\end{equation}
where \( p_n\left( \frac{x}{1 + \frac{\beta}{2}x} \right) \) denotes the element of \( \widehat{\Lambda} \) defined by
\[
p_n\left( \frac{x_1}{1 + \frac{\beta}{2}x_1}, \frac{x_2}{1 + \frac{\beta}{2}x_2}, \dots \right).
\]
When \( \beta = 0 \), the function \( p^{(\beta)}_n \) reduces to the usual power sum \( p_n \).

\begin{lemma}
If \( m \) is odd, then \( p_m^{(\beta)} \in G\Gamma_\sharp \).
\end{lemma}

\begin{proof}
Substituting \( x_1 = t \), \( x_2 = \overline{t} \) into~\eqref{eq:p^beta_vs_p}, we obtain
\[
p_m^{(\beta)}(t, \overline{t}, x_3, \dots) = p_m\left( 
\frac{t}{1 + \frac{\beta}{2}t},
\frac{\overline{t}}{1 + \frac{\beta}{2}\overline{t}},
\frac{x_3}{1 + \frac{\beta}{2}x_3}, \dots
\right).
\]
Since
\[
\frac{\overline{t}}{1 + \frac{\beta}{2} \overline{t}} = - \frac{t}{1 + \frac{\beta}{2} t},
\]
and \( p_m(T, -T, X_3, X_4, \dots) = p_m(X_3, X_4, \dots) \) for odd \( m \), we find that \( p_m^{(\beta)}(t, \overline{t}, x_3, \dots) \) is independent of \( t \).
\end{proof}

Let $\Gamma_\sharp$ be the \( \mathbb{Q}(\beta) \)-algebra
\[
\Gamma_\sharp := \left\{ f(x) \in \Lambda_\sharp \;\middle|\; f(t, -t, x_3, x_4, \dots) \text{ is independent of } t \right\}.
\]
By~\eqref{eq:Gamma_beta}, every element of \( \Gamma_\sharp \) can be uniquely expressed as an infinite sum of the form
\[
\sum_{n=0}^\infty \sum_{\substack{|\lambda| = n \\ \lambda_i \text{ odd}}} 
c_\lambda\, p_{\lambda_1} p_{\lambda_2} \dots p_{\lambda_{\ell(\lambda)}}, \qquad (c_\lambda \in \mathbb{Q}(\beta)).
\]
\begin{prop}\label{cor:p_isbasis_ofG}
The substitution \( x_i \mapsto \frac{x_i}{1 + \frac{\beta}{2} x_i} \) induces a \( \mathbb{Q}(\beta) \)-algebra isomorphism
\[
\Gamma_\sharp \to G\Gamma_\sharp.
\]
Each \( p_m \) (\( m \): odd) is sent to \( p_m^{(\beta)} \) under this map.
As a consequence, every element of \( G\Gamma_\sharp \) can be uniquely expressed as an infinite sum of elements in
\[
\mathbb{Q}(\beta)[p_1^{(\beta)}, p_3^{(\beta)}, p_5^{(\beta)}, \dots].
\]
In particular, \( p_1^{(\beta)}, p_3^{(\beta)}, p_5^{(\beta)}, \dots \) are algebraically independent over \( \mathbb{Q}(\beta) \).
\end{prop}

\subsubsection{\( \beta \)-deformed \( q \)-functions \( q_n^{(\beta)} \)}

Another important family contained in \( G\Gamma_\sharp \) is the \textit{\( \beta \)-deformed \( q \)-functions} \( q_n^{(\beta)} = q_n^{(\beta)}(x) \). These are defined via the generating function
\begin{equation}\label{eq:pbeta_to_qbeta}
\begin{aligned}
Q^{(\beta)}(z) = \sum_{n=0}^\infty q_n^{(\beta)}(x) z^n 
= \prod_i \frac{1 - \overline{x_i} z}{1 - x_i z},
\end{aligned}
\end{equation}
and can be regarded as \( K \)-theoretic analogues of the Schur \( q \)-functions (see~\cite[\S III-8]{macbook}).

\begin{lemma}
For any \( n>0 \), we have \( q_n^{(\beta)} \in G\Gamma_\sharp \).
\end{lemma}

\begin{proof}
Substituting \( x_1 = t \), \( x_2 = \overline{t} \) into~\eqref{eq:pbeta_to_qbeta}, we obtain
\[
\sum_{n=0}^\infty q_n^{(\beta)}(t, \overline{t}, x_3, \dots) z^n
= \frac{1 - \overline{t} z}{1 - t z} \cdot \frac{1 - t z}{1 - \overline{t} z} \cdot 
\prod_{i=3}^\infty \frac{1 - \overline{x_i} z}{1 - x_i z}
= \prod_{i=3}^\infty \frac{1 - \overline{x_i} z}{1 - x_i z}.
\]
Since this expression is independent of \( t \), we find that \( q_n^{(\beta)} \) satisfies the \( K \)-\( Q \)-cancellation property.
\end{proof}

We define the function \( \overline{q}_n^{(\beta)} \) by
\[
\overline{q}_n^{(\beta)}(x) := q_n^{(\beta)}(\overline{x}),
\]
which is clearly also an element of \( G\Gamma_\sharp \).
If $\beta=0$, we have $q_n^{(0)}=\overline{q}_n^{(0)}=q_n$.

\begin{lemma}\label{lemma:involution}
The substitution \( x_i \mapsto \overline{x}_i \) defines a \( \mathbb{Q}(\beta) \)-algebra involution on \( G\Gamma_\sharp \), under which
\[
p_m^{(\beta)} \leftrightarrow -p_m^{(\beta)}, \quad
q_n^{(\beta)} \leftrightarrow \overline{q}_n^{(\beta)}, \qquad (m\text{ odd},\; n \geq 1).
\]
\end{lemma}

\begin{proof}
The first relation follows from~\eqref{eq:p^beta_vs_p}, and the second from the definition of \( \overline{q}_n^{(\beta)} \).
\end{proof}

\begin{lemma}\label{lemma:func_rel_Q}
Let \( \overline{Q}^{(\beta)}(z) = \sum_{n=0}^\infty \overline{q}_n^{(\beta)} z^n \) be the generating function of \( \overline{q}_n^{(\beta)} \). Then the following functional equations hold:
\begin{gather}
Q^{(\beta)}(z) \overline{Q}^{(\beta)}(z) = 1, \qquad
\overline{Q}^{(\beta)}(z) = \frac{Q^{(\beta)}(-z - \beta)}{Q^{(\beta)}(-\beta)}.
\end{gather}
\end{lemma}

\begin{proof}
The first equation follows directly from~\eqref{eq:pbeta_to_qbeta}. For the second, we compute:
\[
\begin{aligned}
Q^{(\beta)}(-z - \beta)
&= \prod_i \frac{1 - \overline{x_i}(-z - \beta)}{1 - x_i(-z - \beta)} \\
&= \prod_i \frac{1 + \frac{x_i}{1 + \beta x_i}(z + \beta)}{1 + x_i(z + \beta)} \\
&= \prod_i \frac{1 - \beta \cdot \frac{x_i}{1 + \beta x_i}}{1 + \beta x_i}
\cdot \frac{1 - x_i z}{1 - \overline{x_i} z} \\
&= \prod_i \frac{1 + \beta \overline{x_i}}{1 + \beta x_i}
\cdot \frac{1 - x_i z}{1 - \overline{x_i} z} \\
&= Q^{(\beta)}(-\beta) \cdot \overline{Q}^{(\beta)}(z).
\end{aligned}
\]
\end{proof}

From the identities in Lemma~\ref{lemma:func_rel_Q}, we obtain the recurrence relation
\begin{equation}\label{eq:G_base_relation}
q_n^{(\beta)} + q_{n-1}^{(\beta)} \overline{q}_1^{(\beta)} + q_{n-2}^{(\beta)} \overline{q}_2^{(\beta)} + \cdots + \overline{q}_n^{(\beta)} = 0, \qquad (n > 0),
\end{equation}
as well as the expression
\begin{equation}\label{eq:ovq_to_q}
\overline{q}_k^{(\beta)} = (-1)^k \cdot \frac{
\sum_{l=0}^\infty \binom{k+l}{k} (-\beta)^l q_{k+l}^{(\beta)}
}{
\sum_{l=0}^\infty (-\beta)^l q_l^{(\beta)}
}.
\end{equation}

The denominator in~\eqref{eq:ovq_to_q} expands as
\begin{equation}\label{eq:q0_expansion}
\left(1 + \sum_{l=1}^\infty (-\beta)^l q_l^{(\beta)} \right)^{-1}
= 1 + \beta q_1^{(\beta)} + \beta^2\left\{(q_1^{(\beta)})^2 - q_2^{(\beta)}\right\} + \cdots,
\end{equation}
from which we deduce that \( \overline{q}_k^{(\beta)} \) has the expansion
\begin{equation}\label{eq:overlineq_to_q}
\overline{q}_k^{(\beta)} = (-1)^k q_k^{(\beta)} + \sum_{\substack{
i = (i_1 \geq i_2 \geq \dots \geq i_m) \\ |i| > k
}} \beta^{|i| - k} c_i \cdot q_{i_1}^{(\beta)} q_{i_2}^{(\beta)} \cdots q_{i_m}^{(\beta)},
\end{equation}
where \( c_i \in \mathbb{Z} \) and \( |i| := \sum_k i_k \).

Now, substituting \( n \mapsto 2n \) into~\eqref{eq:G_base_relation} and inserting~\eqref{eq:overlineq_to_q}, we can isolate the term \( q_{2n}^{(\beta)} \) and obtain:
\begin{equation}\label{eq:zenkasiki}
\begin{aligned}
q_{2n}^{(\beta)} &=
q_{2n-1}^{(\beta)} q_1^{(\beta)} - q_{2n-2}^{(\beta)} q_2^{(\beta)} + \cdots
+ (-1)^n q_{n+1}^{(\beta)} q_{n-1}^{(\beta)} - (-1)^n \frac{(q_n^{(\beta)})^2}{2} \\
& \quad + \frac{1}{2} \sum_{|i| > 2n} \beta^{|i| - 2n} C_i \cdot q_{i_1}^{(\beta)} q_{i_2}^{(\beta)} \cdots q_{i_m}^{(\beta)},
\end{aligned}
\end{equation}
where \( C_i \in \mathbb{Z} \).

\subsubsection{Expansion by odd partitions}\label{sec:expansion_by_odd}

By repeatedly applying the recurrence relation~\eqref{eq:zenkasiki}, one can express \( q_{2n}^{(\beta)} \) in terms of \( q_1^{(\beta)}, q_3^{(\beta)}, q_5^{(\beta)}, \dots \). We now provide a proof of this fact.

We introduce an ordering \( \succ \) on monomials in the \( q_n^{(\beta)} \) as follows:
\[
\begin{aligned}
& q_{k_1}^{(\beta)} q_{k_2}^{(\beta)} \dots q_{k_p}^{(\beta)} 
\succ 
q_{l_1}^{(\beta)} q_{l_2}^{(\beta)} \dots q_{l_r}^{(\beta)} 
\qquad (k_1 \geq k_2 \geq \cdots,\; l_1 \geq l_2 \geq \cdots)
\\
& \iff
\begin{cases}
|k| > |l|,\quad \text{or} \\
|k| = |l| \text{ and } (k_1, k_2, \dots) 
\text{ is lexicographically smaller than } (l_1, l_2, \dots)
\end{cases}
\end{aligned}
\]
With respect to this ordering, every monomial appearing on the right-hand side of~\eqref{eq:zenkasiki} is strictly greater than \( q_{2n}^{(\beta)} \). Thus, any monomial involving any of \( q_2^{(\beta)}, q_4^{(\beta)}, q_6^{(\beta)}, \dots \) can always be rewritten as a linear combination of monomials greater in this order.

Since \( \succ \) is a total order and the set of monomials in the \( q_n^{(\beta)} \) forms a well-ordered set under \( \succ \), this replacement procedure can be iterated. In this way, one can eliminate all even-indexed \( q_{2k}^{(\beta)} \) from the right-hand side of~\eqref{eq:zenkasiki}, and write \( q_{2n}^{(\beta)} \) as an \textit{infinite sum} of elements in \( \mathbb{Q}[\beta][q_1^{(\beta)}, q_3^{(\beta)}, q_5^{(\beta)}, \dots] \).

\begin{remark}
We do not have \( q_{2n}^{(\beta)} \in \mathbb{Q}[\beta][q_1^{(\beta)}, q_3^{(\beta)}, q_5^{(\beta)}, \dots] \); the word ``\textit{infinite sum}'' above cannot be improved as finite sum.
In fact, when \( x_1 = x \) and \( x_2 = x_3 = \cdots = 0 \), we have
\[
q_n^{(\beta)} = x^{n-1}(x - \overline{x}) = x^n \cdot \frac{2 + \beta x}{1 + \beta x}.
\]
However, \( q_{2n}^{(\beta)} \) cannot be written as a polynomial in \( q_1^{(\beta)}, q_3^{(\beta)}, q_5^{(\beta)}, \dots \) with coefficients in \( \mathbb{Q}[\beta] \).\footnote{
(Proof.) It suffices to show \( x^{2n} \cdot \frac{2 + \beta x}{1 + \beta x} \notin \mathbb{Q}(\beta)[x \cdot \frac{2 + \beta x}{1 + \beta x}, x^2] \). 
Consider the substitution \( x \mapsto \sqrt{2}/\beta \), which defines a map
$
\mathbb{Q}(\beta)[x \cdot \tfrac{2 + \beta x}{1 + \beta x}, x^2] \to \mathbb{Q}(\beta)(\sqrt{2})
$.
Obviously, the image of this map is contained in \( \mathbb{Q}(\beta) \subsetneq \mathbb{Q}(\beta)(\sqrt{2}) \).
However, the value of \( x^{2n} \cdot \frac{2 + \beta x}{1 + \beta x} \) at \( x = \sqrt{2}/\beta \) equals
\[
\sqrt{2} \cdot (2/\beta)^{2n},
\] 
which is not an element of $\mathbb{Q}(\beta)$.
Hence, we conclude \( x^{2n} \cdot \frac{2 + \beta x}{1 + \beta x} \notin \mathbb{Q}(\beta)[x \cdot \frac{2 + \beta x}{1 + \beta x}, x^2] \).
}
\end{remark}

For a partition \( \lambda \), define
\[
p_\lambda^{(\beta)} := p_{\lambda_1}^{(\beta)} p_{\lambda_2}^{(\beta)} \cdots, \qquad
q_\lambda^{(\beta)} := q_{\lambda_1}^{(\beta)} q_{\lambda_2}^{(\beta)} \cdots.
\]
Let \( \mathcal{OP} \) denote the set of \emph{odd partitions} (partitions whose parts are all odd), and \( \mathcal{SP} \) denote the set of \emph{strict partitions} (partitions with all parts distinct).

Then the above claim can be restated as follows:

\begin{lemma}\label{eq:expansion_by_odd}
For any partition \( \lambda \in \mathcal{P} \), the function \( q_\lambda^{(\beta)} \) can be written as an infinite \( \mathbb{Q}[\beta] \)-linear combination of \( q_\mu^{(\beta)} \) for \( \mu \in \mathcal{OP} \).
\end{lemma}

\subsubsection{Expansion by strict partitions}

By modifying the argument in \S\ref{sec:expansion_by_odd}, we can also prove the following statement concerning strict partitions.

\begin{lemma}\label{eq:expansion_by_strict}
For any partition \( \lambda \in \mathcal{P} \), the function \( q_\lambda^{(\beta)} \) can be written as an infinite \( \mathbb{Z}[\beta] \)-linear combination of \( q_\mu^{(\beta)} \) with \( \mu \in \mathcal{SP} \).
\end{lemma}

\begin{proof}
We introduce a different total order \( \succ' \) on monomials in the \( q_n^{(\beta)} \) as follows:
\[
\begin{aligned}
& q_{k_1}^{(\beta)} q_{k_2}^{(\beta)} \dots q_{k_p}^{(\beta)} 
\succ' 
q_{l_1}^{(\beta)} q_{l_2}^{(\beta)} \dots q_{l_r}^{(\beta)} 
\qquad (k_1 \geq k_2 \geq \cdots,\; l_1 \geq l_2 \geq \cdots)
\\
& \iff
\begin{cases}
|k| > |l|,\quad \text{or} \\
|k| = |l| \text{ and } (k_1, k_2, \dots) 
\text{ is lexicographically greater than } (l_1, l_2, \dots).
\end{cases}
\end{aligned}
\]

If the partition \( \lambda \) has repeated parts, then by substituting the rearranged recurrence
\begin{equation}\tag{\ref{eq:zenkasiki}'}
\begin{aligned}
(q_n^{(\beta)})^2 &= 
2 q_{n+1}^{(\beta)} q_{n-1}^{(\beta)} 
- 2 q_{n+2}^{(\beta)} q_{n-2}^{(\beta)} 
+ \cdots 
+ (-1)^n 2 q_{2n}^{(\beta)} \\
& \quad 
+ (-1)^n \sum_{|i| > 2n} \beta^{|i| - 2n} C_i \cdot 
q_{i_1}^{(\beta)} q_{i_2}^{(\beta)} \dots q_{i_m}^{(\beta)},
\end{aligned}
\end{equation}
one can express \( q_\lambda^{(\beta)} \) as a linear combination of monomials strictly greater than itself with respect to \( \succ' \).
Repeating this reduction leads to an expression for \( q_\lambda^{(\beta)} \) as an infinite \( \mathbb{Z}[\beta] \)-linear combination of \( q_\mu^{(\beta)} \) for strict partitions \( \mu \in \mathcal{SP} \).
\end{proof}

Using the results above, we obtain the following fundamental theorem regarding bases of \( G\Gamma_\sharp \):

\begin{thm}\label{prop:expansion_of_GGamma}
For each of the following five sets \( \mathrm{(i)} \)–\( \mathrm{(v)} \), any element of \( G\Gamma_\sharp \) can be uniquely expressed as an infinite \( \mathbb{Q}(\beta) \)-linear combination of elements from that set:
\[
\mathrm{(i)} \{p_\lambda^{(\beta)}\}_{\lambda \in \mathcal{OP}} ,\quad
\mathrm{(ii)} \{q_\lambda^{(\beta)}\}_{\lambda \in \mathcal{OP}},\quad
\mathrm{(iii)} \{\overline{q}_\lambda^{(\beta)}\}_{\lambda \in \mathcal{OP}},\quad
\mathrm{(iv)} \{q_\lambda^{(\beta)}\}_{\lambda \in \mathcal{SP}},\quad
\mathrm{(v)} \{\overline{q}_\lambda^{(\beta)}\}_{\lambda \in \mathcal{SP}}.
\]
%
\end{thm}

\begin{proof}
The claim for (i) is exactly Proposition~\ref{cor:p_isbasis_ofG}.  
The basis transformation between (i) and (ii) is given by the generating function~\eqref{eq:pbeta_to_qbeta}.  
The result for (iii) follows from (ii) and the involution in Lemma~\ref{lemma:involution}.  
The statement for (iv) is a direct consequence of Lemma~\ref{eq:expansion_by_strict},  
and that for (v) again follows from (iv) and the involution in Lemma~\ref{lemma:involution}.
\end{proof}

\subsection{\( \mathbb{Q}(\beta) \)-basis of \( G\Gamma \)}

In the previous subsection, we constructed bases for \( G\Gamma_\sharp \) in the sense of infinite sums, but our primary interest lies in \( G\Gamma \). 
The \( K \)-\( Q \) functions play an essential role as building blocks for \( G\Gamma \). These functions, denoted \( GQ_\lambda \), form a family of symmetric functions indexed by strict partitions \( \lambda \) (see~\cite{IKEDA201322}). In this section, we focus only on the case where \( \lambda = (n) \) consists of a single row. 
Let \( GQ_n := GQ_{(n)}(x) \).

The \( K \)-\( Q \) function \( GQ_n \) is defined via the generating function~\cite{HUDSON2017115}
\begin{equation}\label{eq:gen_of_GQ}
\sum_{n \in \mathbb{Z}} GQ_n z^n = 
\frac{1}{1 + \beta z^{-1}} \cdot \frac{Q^{(\beta)}(z)}{Q^{(\beta)}(-\beta)}, 
\qquad 
\text{where } \frac{1}{1 + \beta z^{-1}} = 1 - \beta z^{-1} + \beta^2 z^{-2} - \cdots.
\end{equation}  
In particular, for \( n \geq 0 \), we have \( GQ_{-n} = (-\beta)^n \).

From \eqref{eq:gen_of_GQ}, we obtain the expansion
\begin{equation}\label{eq:exp_of_GQ}
\begin{aligned}
GQ_n =
\frac{q_n^{(\beta)} - \beta q_{n+1}^{(\beta)} + \beta^2 q_{n+2}^{(\beta)} - \cdots}
{1 - \beta q_1^{(\beta)} + \beta^2 q_2^{(\beta)} - \cdots}
=
q_n^{(\beta)} + \sum_{\mu \succ (n)} c_\mu q_\mu^{(\beta)}
\qquad (c_\mu \in \mathbb{Z}[\beta]),
\end{aligned}
\end{equation}
which allows us to regard \( GQ_n \) as an element of \( G\Gamma_\sharp \).

\begin{lemma}\label{lemma:finiteness_of_GQ}
Let \( N \) be a positive integer, and define the specialization
\[
GQ_n(x_1, \dots, x_N) := GQ_n(x_1, \dots, x_N, 0, 0, \dots).
\]
Then we have
\[
GQ_n(x_1, \dots, x_N) \in \mathbb{Z}[\beta][x_1, \dots, x_N].
\]
\end{lemma}

\begin{proof}
Subtracting \( (1 + \beta z^{-1})^{-1} \) from both sides of~\eqref{eq:gen_of_GQ}, we obtain
\begin{equation}\label{eq:gen_of_GQ_2}
\sum_{n = 1}^\infty GQ_n z^n = 
\frac{z}{z + \beta} \cdot \frac{Q^{(\beta)}(z) - Q^{(\beta)}(-\beta)}{Q^{(\beta)}(-\beta)}.
\end{equation}
Substituting~\eqref{eq:pbeta_to_qbeta} and rearranging, we find
\begin{equation}\label{eq:exansion_of_GQn1}
\sum_{n = 1}^\infty GQ_n z^n =
z \cdot 
\frac{
\prod_i (1 + x_i(z + \beta)) \prod_i (1 + \beta x_i) - \prod_i (1 - x_i z)
}{z + \beta}
\cdot 
\frac{1}{\prod_i (1 - x_i z)}.
\end{equation}
Note that the numerator vanishes at \( z = -\beta \):
\[
\left[
\prod_i (1 + x_i(z + \beta)) \prod_i (1 + \beta x_i) - \prod_i (1 - x_i z)
\right]_{z = -\beta} = 0.
\]
Therefore, by the factor theorem, the rational function
\[
\frac{
\prod_{i=1}^N (1 + x_i(z + \beta)) \prod_{i=1}^N (1 + \beta x_i) - \prod_{i=1}^N (1 - x_i z),
}{z + \beta}
\]
obtained after the substitution \( x_{N+1} = x_{N+2} = \cdots = 0 \), is actually a polynomial in \( z \) with coefficients in \( \mathbb{Z}[\beta][x_1, \dots, x_N] \) .  
Hence, the right-hand side of~\eqref{eq:exansion_of_GQn1} is a series in \( z \) with coefficients in \( \mathbb{Z}[\beta][x_1, \dots, x_N] \).
\end{proof}

\begin{exa}
When \( x_1 = x \) and \( x_2 = x_3 = \cdots = 0 \), the generating function~\eqref{eq:exansion_of_GQn1} becomes
\[
\sum_{n = 1}^\infty GQ_n(x) z^n = 
(2x + \beta x^2) \cdot z(1 + xz + x^2 z^2 + \cdots).
\]
In particular, \( GQ_n(x) = x^n (2 + \beta x) \in \mathbb{Z}[\beta][x] \).
\end{exa}

From \eqref{eq:car_of_GGamma} and Lemma \ref{lemma:finiteness_of_GQ}, we find $GQ_n\in G\Gamma$.
Moreover, the structure of $G\Gamma$ is expressed as follows:

\begin{prop}\label{prop:exp_of_GGamma_in_GQ}
Any element of \( G\Gamma \) can be uniquely expressed as an infinite sum of elements in \( \mathbb{Q}(\beta)[GQ_1, GQ_3, GQ_5, \dots] \).
\end{prop}

\begin{proof}
Fix an odd integer \( m \). From Theorem~\ref{prop:expansion_of_GGamma} and~\eqref{eq:exp_of_GQ}, we know that
\begin{equation}\label{eq:GQ_to_q}
GQ_m = q_m^{(\beta)} + \sum_{\substack{\mu \succ (m) \\ \mu \in \mathcal{OP}}} c_\mu q_\mu^{(\beta)}, \qquad (c_\mu \in \mathbb{Q}(\beta)),
\end{equation}
which gives a unique infinite expansion. Using standard arguments based on upper triangularity with respect to the order \( \succ \), one can solve~\eqref{eq:GQ_to_q} recursively to express \( q_m^{(\beta)} \) as an infinite linear combination of \( GQ_1, GQ_3, \dots \)
\end{proof}

\begin{remark}
It is not true that \( GQ_{2n} \in \mathbb{Q}(\beta)[GQ_1, GQ_3, GQ_5, \dots] \).  
In particular, the infinite sum in Proposition~\ref{prop:exp_of_GGamma_in_GQ} cannot be replaced by a finite one.
For example, suppose \( p_1^{(\beta)} = t/2 \) and \( p_3^{(\beta)} = p_5^{(\beta)} = \cdots = 0 \).  
Then we obtain the expression
\[
GQ_n = 1 -
\left(1 - \beta t + \frac{\beta^2 t^2}{2!} - \cdots + (-1)^{n-1} \frac{\beta^{n-1} t^{n-1}}{(n-1)!} \right)
\cdot e^{\beta t} 
\quad \in \mathbb{Q}(\beta)[t, e^{\beta t}].
\]
From this expression, we show that \( GQ_{2n} \) cannot be written as a polynomial in $ GQ_1$, $GQ_3, \dots $ with coefficients in \( \mathbb{Q}(\beta) \).\footnote{
(Proof.) Let \( f_n := 1 - \beta t + \cdots + (-1)^{n-1} \frac{\beta^{n-1} t^{n-1}}{(n-1)!} \).  
Then, we have $GQ_n=1-f_ne^{t\beta}$.
Since $e^{t\beta}$ is transcendental over $\QQ(\beta)[f_1,f_3,f_5,\dots]$, any element in \( \mathbb{Q}(\beta)[GQ_1, GQ_3, GQ_5, \dots] \) must be uniquely written in the form
\[
c_0 + c_1 e^{\beta t} + c_2 e^{2\beta t} + \cdots + c_N e^{N \beta t}
\qquad
(c_i \text{ are homogeneous degree } i \text{ polynomials in } f_1, f_3, f_5, \dots).
\]
If we wish to express \( GQ_{2n}=1-f_{2n}e^{t\beta} \) as such a polynomial, we need to express \( f_{2n} \) as a \( \mathbb{Q}(\beta) \)-linear combination of \( f_1, f_3, f_5, \dots \); however, it is impossible.
}
\end{remark}

\section{Algebra of the dual \( K \)-\( Q \)-cancellation property}\label{sec:dualKQ}

Just like \( G\Gamma \), the algebra \( g\Gamma \) is a \( K \)-theoretic analogue of \( \Gamma \).  
However, unlike \( G\Gamma \), the elements of \( g\Gamma \) are ordinary symmetric functions (i.e., not infinite sums).  
The two algebras \( G\Gamma \) and \( g\Gamma \) are dual to each other with respect to the bilinear pairing introduced by Nakagawa–Naruse (see Equation~\eqref{eq:def_of_bilin}).
This duality can be described in a clean and unified manner by introducing an alternative set of \( \beta \)-deformed power sums \( p_n^{[\beta]} \), as well as the corresponding \( q \)-functions \( q_n^{[\beta]} \).

\subsection{\( \mathbb{Q}(\beta) \)-basis of \( g\Gamma \)}\label{sec:dualgGamma}

\subsubsection{An alternative \( \beta \)-deformed power sum \( p_n^{[\beta]} \)}

Here we introduce a new family of \textit{\( \beta \)-deformed power sums} \( p_n^{[\beta]} = p_n^{[\beta]}(x) \), distinct from the previously defined ones. They are defined by
\begin{equation}\label{eq:def_of_an_p}
p_n^{[\beta]}(x) = \sum_{k=0}^{n-1} \binom{n}{k} \left( \frac{\beta}{2} \right)^k p_{n-k}(x).
\end{equation}
Each \( p_n^{[\beta]} \) is a symmetric function, and hence belongs to \( \Lambda \) (compare with~\eqref{eq:p^beta_vs_p}).

Equation \eqref{eq:def_of_an_p} can be formally interpreted as
\begin{equation}\label{eq:def_of_an_p_formal}
\text{``\( p_n^{[\beta]}(x) = p_n(x + \tfrac{\beta}{2}) - p_n(\tfrac{\beta}{2}) \)"}
\end{equation}
where the expression is only formal, since \( p_n(\tfrac{\beta}{2}) = (\tfrac{\beta}{2})^n + (\tfrac{\beta}{2})^n + \cdots \) is not an element of \( \Lambda \).
Precisely,~\eqref{eq:def_of_an_p_formal} states that for any \( N \), we have
\begin{equation}\label{eq:betap_condition}
p_n^{[\beta]}(x_1, \dots, x_N) = 
p_n(x_1 + \tfrac{\beta}{2}, \dots, x_N + \tfrac{\beta}{2}) - N\left( \tfrac{\beta}{2} \right)^n.
\end{equation}
When \( \beta = 0 \), the function \( p_n^{[\beta]} \) reduces to the ordinary power sum \( p_n \).

\begin{lemma}
If \( m \) is odd, then \( p_m^{[\beta]} \in g\Gamma \).
\end{lemma}

\begin{proof}
It suffices to show that \( p_m^{[\beta]}(x_1, \dots, x_N) \) satisfies the dual \( K \)-\( Q \)-cancellation property for all \( N \).  
Substitute \( x_1 = t \), \( x_2 = -t - \beta \) into~\eqref{eq:betap_condition}. Then, we have
\[
\begin{aligned}
p_m^{[\beta]}(t, -t - \beta, x_3, \dots, x_N) 
&= p_m(t + \tfrac{\beta}{2}, -t - \tfrac{\beta}{2}, x_3, \dots, x_N) - N\left( \tfrac{\beta}{2} \right)^m \\
&= p_m(x_3, \dots, x_N) - N\left( \tfrac{\beta}{2} \right)^m.
\end{aligned}
\]
Since this expression is independent of \( t \), the claim follows.
\end{proof}

\begin{prop}\label{prop:p_isbasis_ofg}
We have
\[
g\Gamma = \mathbb{Q}(\beta)[p_1^{[\beta]}, p_3^{[\beta]}, p_5^{[\beta]}, \dots].
\]
\end{prop}

\begin{proof}
The inclusion 
\[
g\Gamma \supset \mathbb{Q}(\beta)[p_1^{[\beta]}, p_3^{[\beta]}, p_5^{[\beta]}, \dots]
\]
is obvious. We now prove the opposite inclusion.
Any symmetric function with coefficients in \( \mathbb{Q}(\beta) \) can be expressed as a polynomial in the ordinary power sums \( p_1, p_2, \dots \).  
Using~\eqref{eq:def_of_an_p}, we can express them as polynomials in \( p_1^{[\beta]}, p_2^{[\beta]}, \dots \) as well.

Let \( f(x) \in g\Gamma \) be an arbitrary element. Then it can be expressed as a finite sum
\begin{equation}\label{eq:exp_of_f}
f(x) = \sum_\lambda c_\lambda \cdot 
p_{\lambda_1}^{[\beta]} \cdots p_{\lambda_{\ell(\lambda)}}^{[\beta]},
\qquad (c_\lambda \in \mathbb{Q}(\beta)),
\end{equation}
where \( c_\lambda = 0 \) for all but finitely many \( \lambda \).
Since \( f(x) \) satisfies the dual \( K \)-\( Q \)-cancellation property, \eqref{eq:betap_condition} implies that the polynomial
\begin{equation}\label{eq:poly_g}
g(x_1, \dots, x_N) := 
\sum_\lambda c_\lambda \prod_{i=1}^{\ell(\lambda)} 
\left\{ p_{\lambda_i}(x_1, \dots, x_N) - N\left( \tfrac{\beta}{2} \right)^{\lambda_i} \right\}
\end{equation}
must satisfy the ordinary \( Q \)-cancellation property.  
Therefore, we have
\begin{equation}\label{eq:cond_Q_prop}
g(x_1, \dots, x_N) \in \mathbb{Q}(\beta)[p_m(x_1, \dots, x_N)\,;\, m \text{ odd}].
\end{equation}

Let $N$ be a sufficiently large integer so that all $p_{m}(x_1,\dots,x_N)$ in the expression \eqref{eq:poly_g} are  algebraically independent over \( \mathbb{Q}(\beta) \).  
Then \eqref{eq:cond_Q_prop} implies that \( c_\lambda \neq 0 \Rightarrow \lambda \in \mathcal{OP} \).
Hence, we have \( f(x) \in \mathbb{Q}(\beta)[p_1^{[\beta]}, p_3^{[\beta]}, \dots] \), which completes the proof.
\end{proof}

\subsubsection{An alternative \( \beta \)-deformed \( q \)-function \( q_n^{[\beta]} \)}

We define a new family of symmetric functions \( q_n^{[\beta]} \) by the generating function
\begin{equation}\label{eq:Q}
Q^{[\beta]}(z) := \sum_{n=0}^\infty q_n^{[\beta]} z^n
= \prod_i \frac{1 - x_i \overline{z}}{1 - x_i z}.
\end{equation}

\begin{lemma}
For any positive integer \( n \), we have \( q_n^{[\beta]} \in g\Gamma \).
\end{lemma}

\begin{proof}
Substituting \( x_1 = t \), \( x_2 = -t - \beta \) into~\eqref{eq:Q}, we deduce
\[
\begin{aligned}
Q^{[\beta]}(z)\big|_{x_1 = t,\; x_2 = -t - \beta}
&=
\frac{1 - t \overline{z}}{1 - t z} \cdot
\frac{1 - (-t - \beta) \overline{z}}{1 - (-t - \beta) z} \cdot
\prod_{i \geq 3} \frac{1 - x_i \overline{z}}{1 - x_i z} \\
&=
\frac{1 + t \frac{z}{1 + \beta z}}{1 - t z} \cdot
\frac{1 + (-t - \beta) \frac{z}{1 + \beta z}}{1 + t z + \beta z} \cdot
\prod_{i \geq 3} \frac{1 - x_i \overline{z}}{1 - x_i z} \\
&=
\frac{1}{(1 + \beta z)^2} \cdot
\prod_{i \geq 3} \frac{1 - x_i \overline{z}}{1 - x_i z}.
\end{aligned}
\]
Since the result is independent of \( t \), we find that \( q_n^{[\beta]} \) satisfies the dual \( K \)-\( Q \)-cancellation property.
\end{proof}

We define another family of functions \( \overline{q}_n^{[\beta]} \) by
\[
\overline{Q}^{[\beta]}(z) := \sum_{n=0}^\infty \overline{q}_n^{[\beta]} z^n
= \sum_{n=0}^\infty q_n^{[\beta]} \overline{z}^n
= Q^{[\beta]}(\overline{z}).
\]
Using this, we have the expansion:
\begin{equation}\label{eq:qb_to_overqb}
\overline{q}_n^{[\beta]} = (-1)^n \sum_{l=0}^{n-1} \binom{n-1}{l} \beta^l q_{n - l}^{[\beta]}.
\end{equation}
In particular, \( \overline{q}_n^{[\beta]} \in g\Gamma \).

From~\eqref{eq:Q}, we obtain the identity
\begin{equation}\label{eq:QoverQ}
Q^{[\beta]}(z) \cdot \overline{Q}^{[\beta]}(z) =
\prod_i \frac{1 - x_i \overline{z}}{1 - x_i z} \cdot 
\prod_i \frac{1 - x_i z}{1 - x_i \overline{z}} = 1.
\end{equation}
Expanding both sides in \( z \) and comparing the coefficients of \( z^n \), we find the recurrence relation:
\begin{equation}\label{eq:base_relation}
q_n^{[\beta]} + q_{n-1}^{[\beta]} \overline{q}_1^{[\beta]} + 
q_{n-2}^{[\beta]} \overline{q}_2^{[\beta]} + \cdots + 
\overline{q}_n^{[\beta]} = 0.
\end{equation}

For even \( n \), substituting \( n \mapsto 2n \) into~\eqref{eq:base_relation} and inserting~\eqref{eq:qb_to_overqb}, we obtain the expansion:
\begin{equation}\label{eq:q2n_to_qodd}
\begin{aligned}
q_{2n}^{[\beta]} &=
q_{2n-1}^{[\beta]} q_1^{[\beta]} - q_{2n-2}^{[\beta]} q_2^{[\beta]} + \cdots
+ (-1)^n q_{n+1}^{[\beta]} q_{n-1}^{[\beta]} - (-1)^n \frac{(q_n^{[\beta]})^2}{2} \\
& \quad + \frac{1}{2} \sum_{|k| < 2n} \beta^{2n - |k|} C_k \cdot 
q_{k_1}^{[\beta]} q_{k_2}^{[\beta]} \cdots q_{k_p}^{[\beta]},
\end{aligned}
\end{equation}
where \( C_k \in \mathbb{Z} \).

For a partition \( \lambda \), define
\[
p_\lambda^{[\beta]} := p_{\lambda_1}^{[\beta]} p_{\lambda_2}^{[\beta]} \cdots, \qquad
q_\lambda^{[\beta]} := q_{\lambda_1}^{[\beta]} q_{\lambda_2}^{[\beta]} \cdots.
\]
Using the recurrence~\eqref{eq:q2n_to_qodd} and arguments completely analogous to those in the proof of Theorem~\ref{prop:expansion_of_GGamma}, we obtain the following result:

\begin{thm}\label{prop:expansion_of_gGamma}
Each of the following sets \( \mathrm{(i)} \)–\( \mathrm{(v)} \) forms a \( \mathbb{Q}(\beta) \)-basis of \( g\Gamma \):
\[
\mathrm{(i)} \{ p_\lambda^{[\beta]} \}_{\lambda \in \mathcal{OP}},\quad
\mathrm{(ii)} \{ q_\lambda^{[\beta]} \}_{\lambda \in \mathcal{OP}},\quad
\mathrm{(iii)} \{ \overline{q}_\lambda^{[\beta]} \}_{\lambda \in \mathcal{OP}},\quad
\mathrm{(iv)} \{ q_\lambda^{[\beta]} \}_{\lambda \in \mathcal{SP}},\quad
\mathrm{(v)}\{ \overline{q}_\lambda^{[\beta]} \}_{\lambda \in \mathcal{SP}}.
\]
\end{thm}

\subsection{Application: Nakagawa–Naruse's bilinear form}\label{sec:dual_speces}

We define a nondegenerate bilinear form on \( G\Gamma_\sharp \otimes_{\mathbb{Q}(\beta)} g\Gamma \) by
\begin{equation}\label{eq:def_of_bilin}
G\Gamma_\sharp \otimes_{\mathbb{Q}(\beta)} g\Gamma \to \mathbb{Q}(\beta); \qquad 
f \otimes g \mapsto \langle f, g \rangle,
\end{equation}
where the pairing is defined on basis elements by
\[
\langle p_\lambda^{(\beta)}, p_\mu^{[\beta]} \rangle = 2^{-\ell(\lambda)} z_\lambda \delta_{\lambda,\mu}, \qquad \forall \lambda, \mu \in \mathcal{OP},
\]
with \( z_\lambda = \prod_{i \geq 1} i^{m_i} \cdot m_i! \), and \( m_i = m_i(\lambda) = \#\{ s \mid \lambda_s = i \} \).
When \( \beta = 0 \), this bilinear form reduces to the canonical \( \mathbb{Q} \)-valued bilinear form on \( \Gamma \) (see~\cite[\S III-8]{macbook}):
\[
\Gamma \otimes_{\mathbb{Q}} \Gamma \to \mathbb{Q}; \qquad 
f \otimes g \mapsto \langle f, g \rangle, \qquad 
\langle p_\lambda, p_\mu \rangle = 2^{-\ell(\lambda)} z_\lambda \delta_{\lambda,\mu}.
\]

\begin{thm}\label{thm:cauthyK}
We have the identity
\[
\sum_{\lambda \in \mathcal{OP}} 2^{\ell(\lambda)} z_\lambda^{-1} 
p_\lambda^{(\beta)}(x) p_\lambda^{[\beta]}(y) 
= \prod_{i,j} \frac{1 - \overline{x_i} y_j}{1 - x_i y_j}.
\]
In other words, the Cauchy kernel of the bilinear form~\eqref{eq:def_of_bilin} is given by 
\begin{equation}\label{eq:K-cauchy}
\prod_{i,j} \frac{1 - \overline{x_i} y_j}{1 - x_i y_j}.
\end{equation}
\end{thm}

\begin{proof}
We use the well-known identity:
\[
\prod_{i,j} \frac{1 + X_i Y_j}{1 - X_i Y_j}
= \exp\left( 2 \sum_{m=1,3,5,\dots} \frac{p_m(X) p_m(Y)}{m} \right).
\]
We compute the right-hand side of the claimed identity as follows:
\[
\begin{aligned}
\prod_{i,j} \frac{1 - \overline{x_i} y_j}{1 - x_i y_j}
&=
\prod_{i,j}
\left(
\frac{1 + \frac{x_i}{1 + \frac{\beta}{2} x_i}(y_j + \tfrac{\beta}{2})}
     {1 - \frac{x_i}{1 + \frac{\beta}{2} x_i}(y_j + \tfrac{\beta}{2})}
\cdot
\frac{1 - \frac{x_i}{1 + \frac{\beta}{2} x_i} \cdot \frac{\beta}{2}}
     {1 + \frac{x_i}{1 + \frac{\beta}{2} x_i} \cdot \frac{\beta}{2}}
\right) \\
&=
\exp\left( 
2 \sum_{m=1,3,5,\dots} \frac{
p_m\left( \frac{x}{1 + \frac{\beta}{2} x} \right)
\left\{ p_m(y + \tfrac{\beta}{2}) - p_m(\tfrac{\beta}{2}) \right\}
}{m}
\right) \\
&=
\prod_{m=1,3,5,\dots} \sum_{k=0}^\infty 
\frac{2^k}{k! m^k} 
(p_m^{(\beta)})^k (p_m^{[\beta]})^k \\
&=
\sum_{\lambda \in \mathcal{OP}} 
2^{\ell(\lambda)} z_\lambda^{-1} 
p_\lambda^{(\beta)}(x) p_\lambda^{[\beta]}(y).
\end{aligned}
\]
\end{proof}

The function \( \prod_{i,j} \frac{1 - \overline{x_i} y_j}{1 - x_i y_j} \) agrees with the Cauchy kernel for the bilinear form of Nakagawa–Naruse~\cite[\S 6.2]{nakagawa2017generating} 
(see also~\cite[Definition 1.2]{LwMarberg2024} and~\cite[Equation (1)]{iwao2023}).
From this fact, we see that the bilinear form~\eqref{eq:def_of_bilin} coincides with Nakagawa–Naruse's \( K \)-theoretic bilinear form.  
Theorem~\ref{thm:cauthyK} shows that the \( K \)-theoretic Cauchy kernel arises naturally from the use of \( \beta \)-deformed power sums.

\section{The \( \mathbb{Z}[\beta] \)-algebra structure of \( g\Gamma \)}\label{eq:z-str_of_gGamma}

Up to this point, we have studied the algebras mainly over the field \( \mathbb{Q}(\beta) \).  
However, when applying these algebras to geometry, it becomes important to also consider the integrality of their coefficients.
In this final section, we introduce some basic properties of \( g\Gamma \) viewed as a \( \mathbb{Z}[\beta] \)-algebra.

One general method to ensure the integrality of various constants is to provide their combinatorial interpretations.  
For instance, Lewis–Marberg~\cite{LwMarberg2024} gave combinatorial expressions of \( K \)-theoretic symmetric polynomials in terms of Young diagrams; their results proved a Nakagawa-Naruse's conjecture \cite{nakagawa2017generating} on the combinatorial expression of $gp$ and $gq$ functions (to be defined below).

However, in this section we avoid combinatorial arguments and instead present a few results that can be proved using only elementary algebraic computations.

\subsection{The \( gp_n \) functions}

As noted in Section~\ref{sec:KQ}, our main interest lies in \( G\Gamma \) rather than \( G\Gamma_\sharp \).  
Thus, we restrict the bilinear form~\eqref{eq:def_of_bilin} to \( G\Gamma \).  
When we equip \( G\Gamma_\sharp \) with the weak topology induced by~\eqref{eq:def_of_bilin}, the subset \( G\Gamma \subsetneq G\Gamma_\sharp \) becomes a dense subspace.  
Hence, the restriction
\[
G\Gamma \otimes_{\mathbb{Q}(\beta)} g\Gamma \to \mathbb{Q}(\beta)
\]
also defines a nondegenerate bilinear form.

Define the \textit{$gp$ polynomials} \( gp_n = gp_n(x) \) by the generating function
\begin{equation}\label{eq:gp_to_q}
\sum_{n=1}^\infty gp_n z^n = \frac{1}{2 + \beta z} \sum_{n=1}^\infty q_n^{[\beta]} z^n.
\end{equation}
Equivalently, we have
\begin{equation}\label{eq:gp_to_qbeta}
gp_n = \tfrac{1}{2}\left(
q_n^{[\beta]} - \tfrac{\beta}{2} q_{n-1}^{[\beta]} + \tfrac{\beta^2}{4} q_{n-2}^{[\beta]} - \cdots + (-\tfrac{\beta}{2})^{n-1} q_1^{[\beta]}
\right).
\end{equation}

\begin{lemma}
For any positive integer \( N \), the specialization \( gp_n(x_1, \dots, x_N) := gp_n(x)|_{x_{N+1} = x_{N+2} = \cdots = 0} \) is a symmetric polynomial with coefficients in \( \mathbb{Z}[\beta] \).
\end{lemma}

\begin{proof}
From~\eqref{eq:Q}, we know that each \( q_n^{[\beta]}(x_1, \dots, x_N) \) is a symmetric polynomial with coefficients in \( \mathbb{Z}[\beta] \).  
Then from~\eqref{eq:gp_to_qbeta}, it follows that \( gp_n(x_1, \dots, x_N) \in \mathbb{Z}[\beta, \tfrac{1}{2}] \).  
The key question is whether the coefficients actually lie in \( \mathbb{Z}[\beta] \).\footnote{This might seem straightforward, but in fact it is not trivial.}
From~\eqref{eq:Q} and~\eqref{eq:gp_to_q}, we have
\[
\begin{aligned}
&\sum_{n=1}^\infty gp_n(x_1, \dots, x_N) z^n\\
&= \frac{1}{2 + \beta z} \left( 
\prod_{i=1}^N \frac{1 + x_i \cdot \frac{z}{1 + \beta z}}{1 - x_i z} - 1 
\right) \\
&= 
\frac{ 
\prod_{i=1}^N (1 + \beta z + x_i z) 
- \prod_{i=1}^N (1 - x_i z)(1 + \beta z) 
}{2 + \beta z} \cdot 
\frac{1}{\prod_{i=1}^N (1 - x_i z)(1 + \beta z)}.
\end{aligned}
\]
The numerator vanishes at \( z = -\frac{2}{\beta} \), so by the factor theorem, the rational function
\[
\frac{ 
\prod_{i=1}^N (1 + \beta z + x_i z) 
- \prod_{i=1}^N (1 - x_i z)(1 + \beta z)
}{2 + \beta z}
\]
is actually a polynomial in \( x_i, z \) with coefficients in \( \mathbb{Z}[\beta^{\pm 1}] \).  
Since \( -\frac{2}{\beta} \in \mathbb{Z}[\beta^{-1}] \), we find that the coefficients of \( gp_n(x_1, \dots, x_N) \) are contained in \( \mathbb{Z}[\beta^{\pm 1}] \cap \mathbb{Z}[\beta, \tfrac{1}{2}] = \mathbb{Z}[\beta] \).
\end{proof}

\begin{remark}
The function \( gp_n \) plays the role of the dual to \( GQ_n \) in the following sense:
\[
\langle GQ_m, gp_n \rangle = \delta_{m,n}, \qquad (m, n: \text{ odd}),
\]
see~\cite[Equation (35)]{iwao2023}.
\end{remark}

\subsection{The \( \mathbb{Z}[\beta] \)-algebra \( g\Gamma_{\mathbb{Z}} \)}

We define
\[
g\Gamma_{\mathbb{Z}} := \mathbb{Z}[\beta][gp_1, gp_2, gp_3, \dots].
\]
The goal of this subsection is to prove that \( g\Gamma_{\mathbb{Z}} \) is actually generated by the odd-degree \( gp_m \); we will show
\begin{equation}\label{eq:main}
g\Gamma_{\mathbb{Z}} = \mathbb{Z}[\beta][gp_1, gp_3, gp_5, \dots].
\end{equation}

From Proposition~\ref{prop:expansion_of_gGamma} and~\eqref{eq:gp_to_qbeta}, we already know
\[
g\Gamma = 
g\Gamma_{\mathbb{Z}} \otimes_{\mathbb{Z}[\beta]} \mathbb{Q}(\beta) 
= \mathbb{Q}(\beta)[gp_1, gp_3, gp_5, \dots],
\]
which follows by the same arguments used to prove Theorem~\ref{prop:expansion_of_GGamma}.  
The remaining question is whether the expansion of \( gp_{2n} \) in terms of \( gp_1, gp_3, \dots \) involves only coefficients in \( \mathbb{Z}[\beta] \).

We now present a proof of~\eqref{eq:main}.

\begin{lemma}\label{lemma:K-expansion}
Let \( f(z) := \sum_{n=1}^\infty gp_n z^n \), and define \( F(z) := \frac{z f(z)}{1 + f(z)} \). Then the identity
\begin{equation}\label{eq:F=Fbar}
F(z) = F(\overline{z})
\end{equation}
holds.
\end{lemma}

\begin{proof}
From~\eqref{eq:Q} and~\eqref{eq:gp_to_q}, we have \( Q^{[\beta]}(z) = 1 + (2 + \beta z) f(z) \).  
Substituting this into~\eqref{eq:QoverQ}, we have
\begin{equation}\label{eq:rel_of_f(z)}
\left\{ 1 + (2 + \beta z) f(z) \right\} \left\{ 1 + (2 + \beta \overline{z}) f(\overline{z}) \right\} = 1.
\end{equation}
Using \( 2 + \beta \overline{z} = \frac{2 + \beta z}{1 + \beta z} \), we simplify~\eqref{eq:rel_of_f(z)} to obtain
\[
(1 + \beta z) \cdot \frac{f(z)}{1 + f(z)} + \frac{f(\overline{z})}{1 + f(\overline{z})} = 0.
\]
Multiplying both sides by \( -\overline{z} = \frac{z}{1 + \beta z} \), we obtain
\[
\frac{z f(z)}{1 + f(z)} - \frac{\overline{z} f(\overline{z})}{1 + f(\overline{z})} = 0,
\]
which proves~\eqref{eq:F=Fbar}.
\end{proof}

From \eqref{eq:F=Fbar}, $F(z)$ can be regarded as a ``$K$-theoretic even function.''
We call any function \( F(z) \) satisfying \( F(z) = F(\overline{z}) \) a \textit{\( K \)-even function}.
If a \( K \)-even function $F(z)$ has an expansion
\[
F(z) = c_p z^p + c_{p+1} z^{p+1} + \cdots \quad (c_p \in \mathbb{Q}(\beta)),
\]
then the leading exponent \( p \) must be even.  
This follows immediately by comparing the leading terms in the identity \( F(z) = F(\overline{z}) \).

\begin{lemma}\label{lemma:expansion_of_Keven}
Suppose a \( K \)-even function has an expansion
\[
F(z) = c_1 z^2 + c_2 z^3 + \cdots.
\]
Then, each coefficient \( c_{2k} \) can be expressed as a polynomial in \( c_1, c_3, c_5, \dots, c_{2k-1} \) with coefficients in \( \mathbb{Z}[\beta] \).
\end{lemma}

\begin{proof}
Define
$
F_1(z) := F(z) - c_1\beta^{-1} (z + \overline{z})
$.
Using the equation \( z + \overline{z} = \frac{\beta z^2}{1 + \beta z} \), we expand $F_1(z)$ as 
\[
F_1(z) = (c_2 + c_1 \beta) z^3 + (c_3 - c_1 \beta^2) z^4 + (c_4 + c_1 \beta^3) z^5 + \cdots.
\]
Since \( F_1(z) \) is also \( K \)-even, we must have \( c_2 = -c_1 \beta \).
Next, define \( F_2(z) := F_1(z) - (c_3 - c_1 \beta^2) \beta^{-2}(z + \overline{z})^2 \).  
Expanding $F_2(z)$, we have
\[
F_2(z) = (c_4 - c_1 \beta^3 + 2 c_3 \beta) z^5 
+ (c_5 + 2 c_1 \beta^4 - 3 c_3 \beta^2) z^6 
+ (c_6 - 3 c_1 \beta^5 + 4 c_3 \beta^3) z^7 + \cdots.
\]
Again, since \( F_2(z) \) is \( K \)-even, we obtain \( c_4 = c_1 \beta^3 - 2 c_3 \beta \).

For general $k>1$, define the $K$-even functions $F_k(z)$ by the recursive formula $F_{k+1}(z)=F_k(z)-\alpha_{2k+2}\beta^{-k-1}(z+\overline{z})^{k+1}$, where $\alpha_{2k+2}$ is a polynomial in $c_1,c_3,\dots,c_{2k+1}$ defined by the expansion $F_k(z)=\alpha_{2k+2}z^{2k+2}+\alpha_{2k+3}z^{2k+3}+\cdots$.
Since $F_{k+1}(z)$ is $K$-even, we obtain the expression
\[
c_{2k+2}=
(\text{a polynomial in } c_1,c_3,\dots,c_{2k+1}\text{ with coefficients in }\ZZ[\beta]).
\] 
By iterating this process, we obtain for all \( k \geq 1 \) that \( c_{2k} \) is a \( \mathbb{Z}[\beta] \)-polynomial in \( c_1, c_3, \dots, c_{2k-1} \).
\end{proof}

\begin{thm}
The identity~\eqref{eq:main} holds. That is, we have
\[
g\Gamma_{\mathbb{Z}} = \mathbb{Z}[\beta][gp_1, gp_3, gp_5, \dots].
\]
\end{thm}

\begin{proof}
It suffices to prove that \( gp_{2n} \) can be expressed as a \( \mathbb{Z}[\beta] \)-polynomial in \( gp_1, gp_3, \dots, gp_{2n-1} \).
Let \( f(z) = \sum_{n=1}^\infty gp_n z^n \) and \( F(z) = \frac{z f(z)}{1 + f(z)} \) be the functions introduced in Lemma \ref{lemma:K-expansion}.
Consider the formal power series $\varphi(z)$ defined by
\[
\varphi(z) := \frac{F(z)}{z} = \frac{f(z)}{1 + f(z)} = c_1 z + c_2 z^2 + c_3 z^3 + \cdots.
\]
Applying Lemma~\ref{lemma:expansion_of_Keven} to the $K$-even function $F(z)=c_1z^2+c_2z^3+\cdots$, we find that \( c_{2k} \) is a \( \mathbb{Z}[\beta] \)-polynomial in \( c_1, c_3, \dots, c_{2k-1} \).
From the identity
\begin{equation}\label{eq:phi_to_f}
f(z) = \sum_{n=1}^\infty gp_n z^n = \frac{\varphi(z)}{1 - \varphi(z)} 
= \varphi(z) + \varphi(z)^2 + \varphi(z)^3 + \cdots,
\end{equation}
it follows that the coefficient of \( z^{2n+1} \) in this expansion is of the form
\[
gp_{2n+1} = c_{2n+1} + 
(\text{a polynomial in } c_1,c_3,\dots,c_{2n-1}\text{ with coefficients in }\ZZ[\beta]).
\]
Since \( c_2, c_4, \dots, c_{2n} \) are polynomials in \( c_1, c_3, \dots, c_{2n-1} \), we deduce:
\[
c_{2n+1} = gp_{2n+1} + 
(\text{a polynomial in } gp_1,gp_3,\dots,gp_{2n-1}\text{ with coefficients in }\ZZ[\beta]).
\]
Substituting this into the expression for \( f(z) \), we find that each \( gp_{2n} \) is a polynomial in \( gp_1, gp_3, \dots, gp_{2n-1} \) with coefficients in \( \mathbb{Z}[\beta] \).
\end{proof}

\begin{exa}[Calculation of $gp_2$ and $gp_4$]
Repeating the procedure in the proof of Lemma \ref{lemma:expansion_of_Keven}, we obtain the sequence 
\begin{equation}\label{eq:seq_of_c}
(c_2,c_4,c_6,\dots)=(-c_1\beta,c_1\beta^3-2c_3\beta ,-3c_1\beta^5+5c_3\beta^3-3c_5\beta,\dots).
\end{equation}
From \eqref{eq:phi_to_f}, we obtain the expression
\[
\sum_{n=1}^\infty gp_nz^n=\sum_{\substack{|\lambda|=n\\n\geq 1}}
\binom{\ell(\lambda)}{m_1,m_2,\dots,m_{\ell(\lambda)}}c_\lambda z^n,
\]
where \( m_i = m_i(\lambda) = \#\{ s \mid \lambda_s = i \} \), $c_{\lambda}=\prod_{i=1}^{\ell(\lambda)}c_{\lambda_i}$, and 
\[
\binom{a}{b_1,b_2,\dots,b_\ell}=\frac{a!}{b_1!\cdot b_2!\cdots b_\ell!\cdot (a-b_1-b_2-\dots-b_\ell)!}.
\]
Therefore, we obtain the series of equations
\[
\begin{gathered}
gp_1=c_1,\quad gp_2=c_1^2+c_2,\quad gp_3=c_1^3+2c_1c_2+c_3,\quad
gp_4=c_1^4+3c_1^2c_2+2c_1c_3+c_2^2+c_4.
\end{gathered}
\]
From these equations and \eqref{eq:seq_of_c}, we obtain
\[
\begin{gathered}
c_1=gp_1,\quad 
c_2=-gp_1\beta,\quad 
gp_2=gp_1^2-gp_1\beta,\\
c_3=gp_3-gp_1^3+2gp_1^2\beta,\quad 
c_4=-2gp_3\beta+2gp_1^3\beta-4gp_1^2\beta^2+gp_1\beta^3,\\
gp_4=-gp_1^4+2gp_1gp_3+3gp_1^3\beta-2gp_3\beta-3gp_1^2\beta^2+gp_1\beta^3.
\end{gathered}
\]
\end{exa}



\end{document}